\documentclass[11pt]{amsart}
\usepackage{amssymb}

\title[Bi-elliptic Surface]{Smooth Quotients of bi-elliptic surfaces}
\author[Hisao Yoshihara]{}

\newtheorem{theorem}{Theorem}

\newtheorem{corollary}[theorem]{Corollary}
\newtheorem{claim}[theorem]{Claim}

\theoremstyle{definition}
\newtheorem{definition}[theorem]{Definition}
\newtheorem{example}[theorem]{Example}
\theoremstyle{remark}
\newtheorem{remark}[theorem]{Remark}

\begin{document}
\maketitle

\begin{center}

{\sc Hisao Yoshihara}\\
\medskip
{\small{\em Department of Mathematics, Faculty of Science, Niigata University,\\
Niigata 950-2181, Japan}\\
Telephone and fax numbers:  +81 25 262 7443 \\
E-mail:{\tt yosihara@math.sc.niigata-u.ac.jp}}
\end{center}

\bigskip

\begin{abstract}
We consider the quotient $X$ of bi-elliptic surface by a finite automorphism group.  If $X$ is smooth, then it is a bi-elliptic surface or 
ruled surface with irregularity one. As a corollary any bi-elliptic surface cannot be Galois covering of projective plane, hence does not have 
any Galois embedding.


\end{abstract}

\bigskip

\section{Statement of result}
We consider a covering of surface, i.e., let $X_1$ and $X_2$ be connected normal complex surfaces and 
$\pi : X_1 \longrightarrow X_2$ a finite surjective proper holomorphic map. 
There are lots of studies of the covering. We are interested in the following cases (1) and (2):
\begin{enumerate} 
\item $X_2$ is a rational surface, in particular the projective plane $\mathbb P^2$, for example \cite{fl} 
\item $X_1$ has Kodaira dimension $0$, for example \cite{mu, y2}
\end{enumerate} 
In this note we consider the case where $X_1$ is a bi-elliptic surface and $\pi$ is a Galois covering. The definition of bi-elliptic surface is as follows, which has been called a hyperelliptic surface \cite{su}.

\begin{definition}
A bi-elliptic surface is a surface with the geometric genus zero and having an abelian surface as its unramified covering.
\end{definition}

First we note the following: 

\begin{remark}\label{4}
Let $S$ be a bi-elliptic surface.  If $\pi : S \longrightarrow X$ is a 
Galois covering and $X$ is smooth, then $X$ has no curve with negative self-intersection number. In particular, if $X$ is rational, 
then it is $\mathbb P^2$ or $\mathbb P^1 \times \mathbb P^1$.   
\end{remark}

\begin{proof}
Suppose the contrary. Then, there exists an irreducible curve $C$ in $X$ such that $C^2 <0$. We have that $(\pi^*(C))^2=\deg \pi \cdot C^2 <0$. 
Thus there exists an irreducible component $C'$ in $\pi^*(C)$ with $C'^2 < 0$. which is a contradiction. Because there exists no curve with negative self-intersection number in bi-elliptic surface.  
\end{proof}

Our result is stated as follows:

\begin{theorem}\label{1}
Let $S$ be a bi-elliptic surface and $G$ a finite subgroup of $\mathrm{Aut}(S)$, the automorphism group of $S$. 
Let $X=S/G$ be the quotient surface. If $X$ is smooth, then it is a bi-elliptic surface or a ruled surface with irregularity one.  
\end{theorem}

As a corollary we have the following. 

\begin{corollary}
Bi-elliptic surface cannot be a Galois covering of any smooth rational surface. 
\end{corollary}

Miranda \cite{mi} presents a construction of bi-elliptic surface as covering of  rational elliptic  
surface. 
We  present the other examples. 

\begin{example}
Let $A$ be the abelian surface with the period matrix 
\[
\left(\begin{array}{@{\,}rrrc@{\,}}
1 & 0 & \lambda & 0 \\
0 & 1 & 0 & e_n
\end{array}\right), 
\]
where $ \Im \lambda >0$ and $e_{n}=\exp(2\pi\sqrt{-1}/n ), \ (n=4,\ 6)$.
Let $\sigma_i \ (i=1,\ 2)$ be the automorphism of $A$ induced by 
\[
\sigma_1z= \left(\begin{array}{@{\,}rc@{\,}}
1 & 0 \\
0 & e_{n}
\end{array}\right)z + \left(\begin{array}{@{\,}r@{\,}}
1/n \\
0
\end{array}\right)  \ \mathrm{and} \ \sigma_2z =\left(\begin{array}{@{\,}rc@{\,}}
1 & 0 \\
0 & e_{n}
\end{array}\right)z,  
\]
where $z \in \mathbb C^2$. Then $S=A/\langle \sigma_1^m \rangle \ (n=2m)$ and $X=A/\langle \sigma_1 \rangle$  are bi-elliptic surfaces. Note 
that $\pi : S \longrightarrow X$ is a double covering. On the other hand, we have $\sigma_1^m \sigma_2=\sigma_2 \sigma_1^m$, hence $\sigma_2$ induces an automorphism $\overline{\sigma}_2$ on $S$. 
It is easy to see that $S/\langle \overline{\sigma}_2 \rangle$ is isomorphic to $E \times \mathbb P^1$, where $E$ is an elliptic curve with the period matrix $(1/2, \lambda)$. 
\end{example}

\begin{remark}\cite[Example 2.4]{y1}\label{3}
Let $A_i \ (i = 1, 2)$ be the abelian surface defined by the following period matrix:
 \[
  \begin{array}{c}
   \Omega_1 \ = \ \left(
                 \begin{array}{cccc}
                 1 & 0 & \omega & 0 \\
                 0 & 1 & 0 & \omega
                 \end{array}
                 \right) \ {\rm and} \ \  
   \Omega_2 \ = \ \left(
                 \begin{array}{cccc}
                 1 & 0 & (\omega - 1)/3 & 0 \\
                 0 & 1 & (\omega - 1)/3 & \omega 
                 \end{array}
                 \right), \ {\rm respectively}, 
  \end{array}
 \]
  where $\omega = \exp(2\pi \sqrt{-1}/3)$. Let $\sigma_i$ be the automorphism of $A_i$ defined by 
\[
\sigma_1z= \left(\begin{array}{@{\,}rc@{\,}}
1 & 0 \\
0 & \omega
\end{array}\right)z + \left(\begin{array}{@{\,}r@{\,}}
(\omega + 2)/3 \\
0
\end{array}\right)  \ \ \mathrm{and} \ \ \sigma_2z =\left(\begin{array}{@{\,}rc@{\,}}
1 & 0 \\
0 & \omega
\end{array}\right)z,  
 + \left(\begin{array}{@{\,}r@{\,}}
1/3 \\
0
\end{array}\right) 
\]

Then $\sigma_{i}^{3} = id$ on $A_i$ and $S_i = A_i/\sigma_i$ is a bi-elliptic surface. Moreover letting 
\[
\tau_1z= \left(\begin{array}{@{\,}rc@{\,}}
\omega & 0 \\
0 & \omega
\end{array}\right)z \ \ \mathrm{and} \ \ \sigma_2z =\left(\begin{array}{@{\,}rc@{\,}}
\omega & 0 \\
0 & \omega
\end{array}\right)z 
 + \left(\begin{array}{@{\,}r@{\,}}
2/3 \\
0
\end{array}\right) 
\]
we see that $\tau_i$ defines an automorphism of $S_i$ and $S_i/\tau_i$ turns out to be a rational surface.  
Note that the rational surfaces in these examples have singular points. 
\end{remark}

We have defined Galois embedding of algebraic variety and applied it to study several varieties. 
In particular we have shown that many abelian surfaces have Galois embeddings \cite{y2}. 
However, as a corollary of Theorem \ref{1}, we have the following.  

\begin{corollary}\label{2} 
There does not exist any Galois embedding of bi-elliptic surface. 
\end{corollary}

\section{Proof} 
Let $S$ be a bi-elliptic surface and $\pi_1 : S \longrightarrow X$ a finite Galois covering. 
Then $S$ can be expressed  as $\mathbb C^2/\Gamma_1$, where $\Gamma_1$ is the complex crystallographic group given in \cite[Theorem]{su}. 
Let $A=\mathbb C^2/\Gamma_0$ be the abelian surface, where $\pi_2 : A \longrightarrow S$ is an unramified covering such that  
$\Gamma_0$ is the normal subgroup of $\Gamma_1$ and $|\Gamma_1 : \Gamma_0|=2, 3, 4 \ \mathrm{or} \  6 $. 
Put $n=|\Gamma_1 : \Gamma_0 |$. 
Let $G$ be the finite subgroup of $\mathrm{Aut}(S)$ such that $X=S/G$. 
For $g \in G$, we let $\widetilde{g}$ be the lift of $g$ on the universal covering $\mathbb C^2$. 

\begin{claim}
The $\widetilde{g}$ turns out to be an affine transformation. 
\end{claim}

\begin{proof} 
For each $s \in \Gamma_0^n$, there exists $t \in \Gamma_0$ such that $s=t^n$. Since $\Gamma_0 \subset \Gamma_1$ 
and $\Gamma_1$ is a discrete group, 
we have $u \in \Gamma_1$ such that $\widetilde{g}t=u\widetilde{g}$. 
Hence we get $\widetilde{g}t^n=u^n\widetilde{g}$ and $u^n \in \Gamma_0$. 
This implies $\widetilde{g}s \equiv \widetilde{g} (\hspace{-2mm} \mod \Gamma_0)$. 
Let $\mathcal L_0$ be the lattice defining the abelian surface $A=\mathbb C^2/\Gamma_0$. 
Putting $z={}^t(z_1, z_2) \in \mathbb C^2$,  $\widetilde{g}(z_1, z_2)=(\widetilde{g}_1(z_1, z_2), \ \widetilde{g}_2(z_1, z_2))$, we have  
\[
\widetilde{g}_i(z_1+\alpha_1, z_2+\alpha_2)=\widetilde{g}_i(z_1, z_2)+\beta_i, 
\]
where $i=1, 2$, \ 
${}^t(\alpha_1, \alpha_2) \in n \mathcal L_0$  and ${}^t(\beta_1, \beta_2) \in \mathcal L_0$. 
Hence we get 
\[
s^* \left( \frac{\partial \widetilde{g_i}}{\partial z_j} \right) = \frac{\partial \widetilde{g_i}}{\partial z_j} \ \ \ (i, j = 1, 2). 
\]
This means $\partial \widetilde{g_i}/\partial z_j$ is a holomorphic function on $\mathbb C^2/\Gamma_0^n$, which is also an 
abelian surface. Therefore $\partial \widetilde{g_i}/\partial z_j$ is a constant, i.e., $\widetilde{g}$ is an affine transformation.
\end{proof}

Thus $\widetilde{g}$ has the representation $\widetilde{g}z=M(g)z+v(g)$, where $z \in \mathbb C^2, \ M(g) \in GL(2, \mathbb C)$
and $v(g) \in \mathbb C^2$. 
We use this expression hereafter. 
Let $\Gamma_2$ be the affine transformation group generated by $\{ \ \widetilde{g} \ | \ g \in G  \}$ and $\Gamma_1$. 
Then we  have $X=\mathbb C^2/\Gamma_2$. Since $\pi_1 : S \longrightarrow X$ is a Galois covering, we have 
$\Gamma_0$ (resp. $\Gamma_1$) is a normal subgroup of $\Gamma_1$ (resp. $\Gamma_2$) and $\Gamma_2/\Gamma_1 \cong G$.  
Let $\Gamma_1/\Gamma_0$ be generated by $\sigma$. Then, we have the expression 
$\widetilde{\sigma}z=M(\sigma)z+v(\sigma)$, where 
\[
M(\sigma)=\left(\begin{array}{@{\,}rc@{\,}}
1 & 0 \\
0 & e_n
\end{array}\right) \ \ \mathrm{and}\ \  v(\sigma)=   
\left(\begin{array}{@{\,}c@{\,}}
1/n \\
0
\end{array}\right) 
\]. 

\begin{claim}
The $X$ has two fibrations. 
\end{claim}

\begin{proof}
Since $\Gamma_1$ is a normal subgroup of  $\Gamma_2$, there exists an integer $r$ such that $\widetilde{g}\widetilde{\sigma}\widetilde{g}^{-1}=t {\widetilde{\sigma}}^r$, where $g \in G$ and $t \in \Gamma_0$. 
This means $M(g)M(\sigma)M(g)^{-1}=M(\sigma)^r$, hence $r=1$. 
Therefore we have $M(g)M(\sigma)=M(\sigma)M(g)$. If we put $\displaystyle{M(g)=\left(\begin{array}{@{\,}rc@{\,}}
a & b \\
c & d
\end{array}\right)}$, it is easy to see that $M(g)$ is a diagonal matrix. 
Hence the lift of each element $\gamma$ of  $\Gamma_2$ can be expressed as 

\[
\widetilde{\gamma}(z)=
\left(\begin{array}{@{\,}cc@{\,}}
\alpha_1 & 0 \\
0 & \alpha_2
\end{array}\right)z  
 + \left(\begin{array}{@{\,}r@{\,}}
a_1 \\
a_2
\end{array}\right). \eqno{(*)}
\]

Since $\mathcal L_0$ is the lattice defining the abelian surface $A$, we have \\
$\Gamma_0=\{ \ \ell \ | \ \ell(z)=z+{}^t(b_1, b_2), \ \mathrm{where} \  {}^t(b_1, b_2) \in \mathcal{L}_0  \}$. \\
Put $\Gamma_{0i}=\{ \ \ell_i \ | \ \ell_i(z_i)=z_i+b_i, \mathrm{where} \ {}^t(b_1, b_2) \in \mathcal{L}_0  \}$. 
Referring to the list of abelian surfaces in \cite{su}, we infer that $\mathbb C/\Gamma_{0i}$ is an elliptic curve. Using the above representation $(*)$, we define 
\[
\Gamma_{2i}=\{ \ \gamma_i \ | \ \gamma_i(z_i)=\alpha_i z_i+a_i, \mathrm{where} \ \gamma \in \widetilde{\Gamma_2}  \}. 
\]
Then $\Gamma_{0i}$ is a subgroup of $\Gamma_{2i}$ with a finite index. Hence $\Gamma_{0i}$ is a discrete subgroup and 
$\mathbb C/\Gamma_{2i} $ is a smooth curve. Thus we have the following diagram 

\newcommand{\mapright}[1]{%
 \smash{\mathop{%
  \hbox to 1cm{\rightarrowfill}}\limits^{#1}}}
\newcommand{\mapleft}[1]{%
 \smash{\mathop{%
  \hbox to 1cm{\leftarrowfill}}\limits_{#1}}}
\newcommand{\mapdown}[1]{\Big\downarrow
 \llap{$\vcenter{\hbox{$\scriptstyle#1\,$}}$ }}
\newcommand{\mapup}[1]{Big\uparrow
 \rlap{$\vcenter{\hbox{$\scriptstyle#1$}}$ }}
\[ \begin{array}{ccc}
\mathbb C^2 & \mapright{p_i} & \mathbb C \\
\mapdown{\pi} & & \mapdown{\pi_i} \\
X=\mathbb C^2/\Gamma_2 & \mapright{\bar{p_i}} & C_i=\mathbb C/\Gamma_{2i}
\end{array}, 
\] 
where $i=1, 2$.  
Consequently we get two fibrations $\bar{p_i} : X \longrightarrow C_i$. 
\end{proof}

Suppose that $X$ is smooth. Then we have 
\[
\dim \mathrm{H}^1(X,\ \mathcal O_X) \le \dim \mathrm{H}^1(S,\ \mathcal O_S) =1,  \ \mathrm{and} \] 
\[  
\dim \mathrm{H}^2(X,\ \mathcal O_X) \le \dim \mathrm{H}^2(S,\ \mathcal O_S) =0 
\]
and the Kodaira dimension of $X$ is less than or equal to $0$. 
By Remark \ref{4} and the classification theorem of algebraic surfaces, we infer that $X$ is $\mathbb P^1 \times \mathbb P^1$, a ruled surface with 
irregularity one or a bi-elliptic surface. 

\begin{claim}
The $X$ cannot be $\mathbb P^1 \times \mathbb P^1$. 
\end{claim} 

\begin{proof}
In the representation above $(*)$ there exists $i$ such that $\alpha_i \ne 1$. 
Then the fiber space $\bar{p_i} : X \longrightarrow C_i$ has a multiple singular fiber. 
Let $F_i$ be the fiber of the projection $\pi_i : \mathbb P^1 \times \mathbb P^1 \longrightarrow \mathbb P^1$. 
Then each divisor $D$ is linearly equivalent to $n_1F_1 + n_2F_2$. If $D^2=0$, then we have $n_1=0$ or $n_2=0$. 
So that there are only two fibrations from $\mathbb P^1 \times \mathbb P^1$ to $\mathbb P^1$, 
which are the first and second projections. 
Neither of projections have any multiple fiber.  
Combining the above assertions we infer readily the conclusion. 
\end{proof}

This completes the proof.

\bibliographystyle{amsplain}

\begin{thebibliography}{99}

\bibitem{fl}M.\ Friedman and M.\ Leyenson,
On ramified covers of the projective plane I: 
Segre's theory and classification in small degrees 
(with an appendix by Eugvenii Shustin) 
{\em arXiv} : 0903.3359v7 [math. AG] 31 Jul 2010.

\bibitem{mi}R.\ Miranda,
bi-elliptic surface as coverings of rational surfaces,  
{\em Advances in Math.}, {\bf 198} (2005) 439--447.

\bibitem{mu}A.\ Muammed Uluda\v g,
Galois coverings of the plane by $K3$ surfaces,
{\em Kyushu J. Math.}, {\bf 59} (2005), 393--419.

\bibitem{su}T.\ Suwa,
On hyperelliptic surfaces,
{\em J.\ Fac.\ Sci.\ Univ.\ Tokyo, Sect., I}, {\bf 16} (1970), 469--476.

\bibitem{y1}H.\ Yoshihara,
Degree of irrationality of hyperelliptic surfaces, 
{\em Algebra Colloquium}. {\bf 7} (2000), 319--328.

\bibitem{y2}\bysame,
Galois embedding of algebraic variety and its application to abelian surface, 
{\em Rend. Sem. Mat. Univ. Padova}. {\bf 117} (2007), 69--86.


\end{thebibliography}

\end{document}